\documentclass[12pt]{amsart}
\usepackage{amssymb}
\usepackage{amscd}

\newtheorem{theorem}{Theorem}[section]

\newtheorem{lemme}[theorem]{Lemma}
\newtheorem{defin}{Definition}[section]

\allowdisplaybreaks \theoremstyle{definition}
\theoremstyle{remark} \numberwithin{equation}{section}

\begin{document}

\title [The uncertainty principle in Clifford analysis]{The uncertainty principle in Clifford analysis\\}%

\author[ J. El Kamel, R. Jday]{ Jamel El Kamel \qquad and \qquad Rim Jday }
 \address{Jamel El Kamel. D\'epartement de Math\'ematiques fsm. Monastir 5000, Tunisia.}
 \email{jamel.elkamel@fsm.rnu.tn}
 \address{Rim Jday . D\'epartement de Math\'ematiques fsm. Monastir 5000, Tunisia.}
 \email{rimjday@live.fr}
\begin{abstract}
In this paper, we provide  the Heisenberg's inequality and the Hardy's theorem for the Clifford-Fourier transform on $\mathbb{R}^m$. 
\end{abstract}
\maketitle
{\it keywords:} Clifford algebra, Clifford-Fourier transform,  Heisenberg's inequality, Hardy's theorem.\\

\noindent MSC (2010) 42B10, 30G35.\\
\section{Introduction}

 In harmonic analysis, the uncertainty principle states that a  non zero function and it's Fourier transform cannot both be very rapidly decreasing. This fact is expressed  by two formulations of the uncertainty principle for the Fourier transform, Heisenberg's inequality and Hardy's theorem.\\
 The classical Fourier transform is defined on $L^1(\mathbb{R})$ by:
$$\mathcal{F}(f)(x)=(2\pi)^\frac{-1}{2}\displaystyle\int_{\mathbb{R}}f(y)e^{-ixy}dy.$$ 
 The Heisenberg's inequality asserts that for $f\in L^2(\mathbb{R})$ 
$$\displaystyle\int_{\mathbb{R}}||x||^2||f(x)||^2dx\displaystyle\int_{\mathbb{R}}||\lambda||^2||\mathcal{F}(f)(\lambda)||^2d\lambda\geq  \frac{1} {4} \left[\displaystyle\int_{\mathbb{R}}||f(x)||^2dx\right]^2$$
with equality only if $f$ almost everywhere equal to $ce^{-px^2}$ for some $p>0.$ The proof of this inequality is  given by Weyl   in \rm{[9]}.\\
 Hardy's theorem \rm{[5]} says : If we  suppose p and q be positive constants and $f$ a function on the real line satisfying   $|f(x)|\leq C e^{-px^2}$ and $|\mathcal{F}(f)(\lambda)|\leq C e^{-q\lambda^2} $ for some positive
 constant C, then (i)  $f=0$ if $pq>\frac{1}{4}$; (ii) $f=A e^{-px^2}$ for some constant $A$ if $pq=\frac{1}{4};$ (iii) there are many $f$ if $pq<\frac{1}{4}.$\\

In  Dunkl theory, the Heisenberg's inequality for the Dunkl transform was proved by R$\rm{\ddot{o}}$sler in \rm{[6]}
 and Shimeno in \rm {[7]}. Hardy's theorem for the Dunkl transform was given by Shimeno in \rm {[7]}.\\

 In Clifford analysis, the Clifford-Fourier  transform was introduced  and studied by De Bie and Xu in \rm{[3]}. In \rm{[4]}, differents properties of the Clifford-Fourier transform and  the Plancherel's theorem were established.\\

 In this paper, we provide an analogues of the Heisenberg's inequality and Hardy's theorem for the Clifford-Fourier transform on $\mathbb{R}^m$.\\ 
 Our  paper is orgonized as follows. In section 2, we review basic notions and notations related to the Clifford algebra. 
  In section 3, we recall some results and properties for the Clifford-Fourier transform useful in the sequel. Also, 
  we provide some new properties associated  to the kernel of the integral Clifford-Fourier transform and the  Clifford-Fourier transform. 
 In section 4, we  prove the Heisenberg's inequality for the  Clifford-Fourier transform. 
  In section 5, we provide the Hardy's theorem in Clifford analysis on $\mathbb{R}^m$, with $m$ even. 

\section{Notations and preliminaries}
 We introduce the Clifford algebra $Cl_{0,m}$  over $\mathbb{R}^m$ as a non commutative algebra generated by the basis $\{e_1,..,e_m\}$ satisfying the rules:\\

 \begin{equation}
 \left\{
\begin{array}{cc}
 \displaystyle  e_ie_j=-e_je_i,\qquad\qquad \text{if}\quad i\neq j; \qquad\\
   \displaystyle e_i^2=-1,\qquad\qquad\quad \forall 1\leq i\leq m. \\

\end{array}
\right.
 \end{equation} 
This algebra can be decomposed as 
  \begin{equation}
 Cl_{0,m}=\oplus_{k=0}^m Cl_{0,m}^k,
  \end{equation}
 with $Cl_{0,m}^k$ the space of vectors defined by
  \begin{equation}
 Cl_{0,m}^k=span\{e_{i_1}..e_{i_k},i_1<..<i_k\}.
  \end{equation}
 Hence, $\{1,e_1,e_2,..,e_{12},..,e_{12..m}\}$ constitute a basis of $Cl_{0,m}.$ \\
 A Clifford number  $x$ in $Cl_{0,m}$ is written as follows :
  \begin{equation}
 x=\displaystyle\sum_{A\in J}x_Ae_A,
  \end{equation} 
 where $J :=\{0,1,..,m,12,..,12..m\},$ $x_A$ real number  and $e_A$ belong to the  basis of $Cl_{0,m}$ defined above.\\
 The norm of such element x is given by :
  \begin{equation}
 {||x||}_c={\left(\displaystyle\sum_{A\in J} x_A^2\right)}^\frac{1}{2}.
  \end{equation}
  In particular, if $x$ is a vector in $Cl_{0,m}$ then
  $$||x||_c^2=-x^2.$$ 
 The Dirac operator is defined by :
  \begin{equation}
  \partial_x=\displaystyle\sum_{i=1}^m e_i\partial_{x_i}.
   \end{equation}
 We introduce respectively the Gamma operator  associated to a vector $x$, the inner product and the wedge product of two vectors $ x$ and $y$ :
  \begin{equation}
 \qquad\qquad \Gamma_x:=-\displaystyle\sum_{j<k}e_je_k(x_j\partial_{x_k}-x_k\partial_{x_j});
  \end{equation}
   \begin{equation}
 \quad <x,y> :=\displaystyle\sum_{j=1}^m x_jy_j=\frac{-1}{2}(xy+yx);
  \end{equation}
   \begin{equation}
\qquad \qquad \qquad\quad  x\wedge y:=\displaystyle\sum_{j<k}e_je_k(x_j y_k-x_ky_j)=\frac{1}{2}(xy-yx).
 \end{equation}
 
 \noindent Every function $f$ defined on $\mathbb{R}^m$ and taking values in $Cl_{0,m}$ can be decomposed as :
  \begin{equation}
 f(x)=f_0(x)+\displaystyle\sum_{i=1}^me_if_i(x)+\displaystyle\sum_{i<j}e_ie_jf_{ij}(x)+..+e_1..e_mf_{1..m}(x),
  \end{equation}
 with $f_0,f_i,..,f_{1..m}$ all real-valued functions.\\
 We denote by :\\ 
 $\bullet \mathcal{P}_k$ the space of homogenious polynomials of degree $k$ taking values in $Cl_{0,m},$\\
$\bullet \mathcal{M}_k:=Ker \partial_{x}\cap \mathcal{P}_k$ the space of spherical monogenics of degree $k,$\\
$\bullet B(\mathbb{R}^m)\otimes Cl_{0,m}$ a class of integrable functions taking values in $Cl_{0,m}$ and satisfying
 \begin{equation}
\int_{\mathbb{R}^m}(1+{||y||}_c)^\frac{m-2}{2}||f(y)||_cdy< \infty,
 \end{equation}
$\bullet L_2(\mathbb{R}^m)\otimes Cl_{0,m}$ the space of square integrable functions  taking values in $Cl_{0,m}$ with the norm equiped
 \begin{equation}
||f||_{2,c}=\left(\displaystyle\int_{\mathbb{R}^m}{||f(x)||}_c^2dx\right)^{\frac{1}{2}}
=\left(\displaystyle\int_{\mathbb{R}^m}\left(\displaystyle\sum_{A\in J}(f_A(x))^2\right)^\frac{1}{2}dx\right)^{\frac{1}{2}},
 \end{equation}
where $J=\{0,1,..,m,12,13,23..,12..m\},$\\
$\bullet S(\mathbb{R}^{m})$ the Schwartz space of infinitely differentiable functions on $\mathbb{R}^{m}$ which are rapidly decreasing as their derivatives.

 \section{Clifford-Fourier Transform}
 \begin{defin}
 Let $f \in  B(\mathbb{R}^{m})\otimes Cl_{0,m}.$ We define the Clifford-Fourier transform $\mathcal{F}_{\pm}$ of $f$  by   (see \rm{[3]}) :
  \begin{equation}
 \mathcal{F}_{\pm}(f)(y)={(2\pi)}^\frac{-m}{2}\int_{\mathbb{R}^m} K_{\pm}(x,y)f(x)dx,
  \end{equation}
 where 
  \begin{equation}
 K_{\pm}(x,y)=e^{\mp i\frac{\pi}{2}\Gamma_y}e^{-i<x,y>}.
  \end{equation}
 \end{defin}
 \begin{lemme}Assume $m$ be even and upper than two.  Then 
  \begin{equation}
{||K_-(x,y)||}_c\leq C e^{{||x||}_c{||y||}_c}, \qquad \forall x,y\in \mathbb{R}^m,
 \end{equation}
with 
C a positive constant.
\end{lemme}
\begin{proof}
Recall that the kernel $K_-(x,y)$ can be decomposed as 
$$K_-(x,y)=K_0^-(x,y)+\displaystyle\sum_{i<j}e_ie_jK_{ij}^-(x,y),$$
with $K_0^-(x,y)$ and  $K_{ij}^-(x,y)$ scalar functions. Moreover, by theorem 5.3 in \rm{[3]} we have 
for $m$ even and $x, y \in\mathbb{R}^m$ :
$$||K_0^-(x,y)||_c\leq c(1+\left\|x\right\|_c)^{\frac{m-2}{2}}(1+\left\|y\right\|_c)^{\frac{m-2}{2}},$$
$$\quad\quad||K_{ij}^-(x,y)||_c\leq c(1+\left\|x\right\|_c)^{\frac{m-2}{2}}(1+\left\|y\right\|_c)^{\frac{m-2}{2}}, i\neq j,$$
which complete the proof.

\end{proof} 
 \begin{lemme}
Let $c>0,$ then 
 \begin{equation}
K_{\pm}(c^{-1}x,cy)= K_{\pm}(x,y), \quad forall x,y \in\mathbb{R}^m.
 \end{equation}
\end{lemme}
\begin{proof} By proposition 3.4 in \rm{[3]}, it is enough to prove the lemma just for $K_-.$\\
As the kernel of the Clifford-Fourier transform is given by (see theorem 3.2 in \rm{[3]})
$$K_{-}(x,y)=A_\lambda(w,z)+B_\lambda(w,z)+x\wedge y C_\lambda(w,z)$$
where $$A_\lambda(w,z)=\displaystyle 2^{\lambda-1}\Gamma(\lambda+1)\displaystyle\sum_{k=0}^{\infty}(i^{2\lambda+2}+(-1)^k)z^{-\lambda}J_{k+\lambda}(z)C_k^{\lambda}(w),\qquad$$
$$B_\lambda(w,z)=-2^{\lambda-1}\Gamma (\lambda )\displaystyle\sum_{k=0}^{\infty}(k+\lambda)(i^{2\lambda+2}-(-1)^k)z^{-\lambda}J_{k+\lambda}(z)C_k^{\lambda}(w),$$
$$\quad C_\lambda(w,z)=-2^{\lambda-1}\Gamma(\lambda)\sum_{k=0}^{\infty}(i^{2\lambda+2}+(-1)^k)z^{-\lambda-1}J_{k+\lambda}(z)\left(\frac {d}{dw} C_k^{\lambda}\right)(w),$$
with $z={||x||}_c{||y||}_c,$ $zw=<x,y>$ and $\lambda=\frac{m-2}{2}.$
Thus, it suffies to substitute $x$ and $y$ by $c^{-1}x$ and $cy$.
\end{proof}
\begin{lemme} Let $c$ be positive constant and $f\in B(\mathbb{R}^m)\otimes Cl_{0,m}.$\\
\noindent  Assume $f_c(x):=f(cx)$ for all $ x\in\mathbb{R}^m$. Then 
\begin{equation}
 \mathcal{F}_{\pm} (f_c)(\lambda)=c^m\mathcal{F}_{\pm} (f)(c^{-1}\lambda).
\end{equation}
\end{lemme}
\begin{proof} Applying a variable change to definition 3.1, we get
$$\mathcal{F}_{\pm} (f_c)(\lambda)
=({2\pi)}^\frac{-m}{2}c^m\displaystyle\int_{\mathbb{R}^m}K_{\pm}(\frac{x}{c},\lambda)f(x)dx.$$
By lemma 3.2, we find 
$$\mathcal{F}_{\pm} (f_c)(\lambda)=({2\pi)}^\frac{-m}{2}c^m\displaystyle\int_{\mathbb{R}^m}K_{\pm}(x,c^{-1}\lambda)f(x)dx.$$
\noindent We conclude.
\end{proof}

\begin{theorem}[The case $i=\frac{m-2}{2},$ theorem $6.4$ in \rm{[4]}]
The Clifford-Fourier transform $\mathcal{F}_{\pm}$ can be extended from $\mathcal{S}(\mathbb{R}^m)\otimes Cl_{0,m}$ to a continuous map on $L_2(\mathbb{R}^m)\otimes Cl_{0,m}.$\\
In particular, when $m$ is even, for all $f\in L_2(\mathbb{R}^m)\otimes Cl_{0,m},$ we have :
 \begin{equation}
||\mathcal{F}_{\pm}(f)||_{2,c}=||f||_{2,c}.
 \end{equation}
\end{theorem}
\begin{theorem}
The Clifford-Fourier transform coincides with the classical Fourier transform for radial functions (see\rm{[4]}).\\ 
In particular, 
 \begin{equation}
 \displaystyle
 \mathcal{F}_{-}\left(e^{-\frac{||.||^2}{2}}\right)(x)=e^{-\frac{||x||^2}{2}},\quad x \in\mathbb{R}^m.
 \end{equation}
\end{theorem}
\section{Heisenberg's inequalities }
In this section, we establish the Heisenberg's inequality for the Clifford-Fourier transform. The proof is based on the consideration of a basis of  $L_2(\mathbb{R}^m)\otimes Cl_{0,m}$  which is expressed in terms of Laguerre polynomials and recurrence relation's. Indeed it is an analogue of the proof of Heisenberg's inequality in the Dunkl analysis (see \rm{[7]}).\\
We introduce a basis $\{\psi _{j,k,l}\}$ of $\mathcal{S}(\mathbb{R}^m)\otimes Cl_{0,m}$  which is defined as follows :
 \begin{equation}
\psi_{2j,k,l}(x):=L_j^{\frac{m}{2}+k-1}({||x||}_c^2) M_k^{(l)}e^{-\frac{{||x||}_c^2}{2}},
 \end{equation}
 \begin{equation}
\psi_{2j+1,k,l}(x):=L_j^{\frac{m}{2}+k}({||x||}_c^2)x M_k^{(l)}e^{-\frac{{||x||}_c^2}{2}}.
 \end{equation}

\noindent where $j,k \in \mathbb{N}$, $\{M_k^{(l)}\in \mathcal{M}_k;\quad l=1,.., dim \mathcal{M}_k\}$ is an  orthogonal basis for $\mathcal{M}_k$ and $L_j^\alpha$ the classical Laguerre polynomials.
\begin{lemme}
 $\{\psi _{j,k,l}\}$ is an orthogonal basis of $L_2(\mathbb{R}^m)\otimes Cl_{0,m}.$
\end{lemme}
\begin{proof}
Sommen and Brackx in \rm{[1]} give an orthogonal basis of $L_2(\mathbb{R}^m)\otimes Cl_{0,m}$ in terms of generalized Clifford polynomials as follows:
$$\psi _{j,k,l}(x)=H_{j,k,l}(\sqrt{2}x)M_k^{(l)}(\sqrt{2}x)e^{-\frac{{||x||}_c^2}{2}},$$
with $H_{j,k,l}$ the Hermite polynomials and $M_k^{(l)}\in \mathcal{M}_k$; $l=1,.., dim (\mathcal{M}_k)$; $j,k \in \mathbb{N}.$\\
Using the relations of the Hermite polynomials defined in \rm{[2]}  ,we get the desired result.
\end{proof} 
\begin{theorem} (see \rm{[3]}). For the basis $\{\psi_{j,k,l}\}$ of $\mathcal{S}(\mathbb{R}^m)\otimes Cl_{0,m},$ one has
\begin{equation}
\mathcal{F}_{\pm}(\psi_{2j,k,l})={(-1)}^{j+k}(\mp1)^k\psi_{2j,k,l},
\end{equation}
\begin{equation}
\qquad\quad\mathcal{F}_{\pm}(\psi_{2j+1,k,l})=i^m {(-1)}^{j+1}(\mp1)^{k+m-1}\psi_{2j+1,k,l}.
 \end{equation}

\end{theorem}

\begin{lemme}
Let $a,b,c$ are positive real numbers such that $${(\frac{a}{t})}^2+b^2t^2\geq c, \forall t>0,$$
then $$ab\geq \frac{c}{2}.$$
\end{lemme}
\begin{proof}
Let $$h:\mathbb{R}_+^*\rightarrow\mathbb{R}$$
$$\qquad\qquad\qquad\quad t\mapsto{(\frac{a}{t})}^2+b^2t^2.$$
It is clear that h is a differentiable function and the corresponding derivative is given by :
$$h'(t)=2b^2t-2a^2t^{-3}$$
$$\qquad\quad =2t^{-3}(b^2t^4-a^2)$$
$$\qquad\qquad \qquad=2t^{-3}(bt^2-a)(bt^2+a).$$
Since $h(t)\geq c,\forall t>0,$ then 
$$\displaystyle 2ab=h(\sqrt{\frac{a}{b}})=\inf_{t>0} h(t)\geq c. $$
\end{proof}
\begin{theorem} Let $f\in L_2(\mathbb{R}^m)\otimes Cl_{0,m}$. Then
\begin{equation}
\displaystyle\int_{\mathbb{R}^m}||x||_c^2||f(x)||_c^2dx\displaystyle\int_{\mathbb{R}^m}||\lambda||_c^2||\mathcal{F}(f)(\lambda)||_c^2d\lambda\geq  \frac{m^2} {4} \left[\displaystyle\int_{\mathbb{R}^m}||f(x)||_c^2dx\right]^2,
\end{equation}
with equality if $f=ce^{-p||x||_c^2},$ almost everywhere for some $p>0$.
\end{theorem}
\begin{proof} We start by proving the inequality .\\
The application of the recurrence  relation for the Laguerre polynomials  
$$tL_n^{(\alpha)}(t)=-(n+1)L_{n+1}^{(\alpha)}(t)+(\alpha+2n+1)L_n^{(\alpha)}(t)-(\alpha+n)L_{n-1}^{(\alpha)}(t)$$
to the basis $\{\psi _{j,k,l}\}$ and relations (4.1) and (4.2)
leads to :\\
\begin{equation}
||x||_c^2\psi_{2j,k,l}(x)
=-(j+1)\psi_{2(j+1),k,l}+(\frac{m}{2}+k+2j)\psi_{2j,k,l}-(\frac{m}{2}+k-1+j)\psi_{2(j-1),k,l}
\end{equation}
\begin{equation}
||x||_c^2\psi_{2j+1,k,l}(x)
=-(j+1)\psi_{2j+3,k,l}+(\frac{m}{2}+k+2j+1)\psi_{2j+1,k,l}-(\frac{m}{2}+k+j)\psi_{2j-1,k,l}.
\end{equation}
We set $\psi_{-1,k,l}=0$ and $\psi_{-2,k,l}=0.$\\

\noindent As $\{\psi_{j,k,l}\}$ is a basis of $L_2(\mathbb{R}^m)\otimes Cl_{0,m}$, every function $f$ in $L_2(\mathbb{R}^m)\otimes Cl_{0,m}$  can be writen as follows :
\begin{equation}
\displaystyle f=\sum_{j,k,l}a_{j,k,l}\psi_{j,k,l}=\sum_{2j,k,l}a_{2j,k,l}\psi_{2j,k,l}+\sum_{2j+1,k,l}a_{2j+1,k,l}\psi_{2j+1,k,l},
\end{equation}
where $a_{j,k,l}=<f,\psi_{j,k,l}>.$\\
 Theorem 4.2 implies that :
 \begin{equation}
\displaystyle \mathcal{F}_{\pm}(f)=\sum_{j,k,l}a_{2j,k,l}(-1)^{j+k}{(\mp1)}^k\psi_{2j,k,l}+\sum_{j,k,l}a_{2j+1,k,l}i^m(-1)^{j+1}{(\mp1)}^{k+m-1}\psi_{2j+1,k,l}.
\end{equation}
Using the last  relations (4.6) and (4.7) and the ortogonality of the basis $\{\psi_{j,k,l}\}$ in $L_2(\mathbb{R}^m)\otimes Cl_{0,m}$, we get : 

$$\displaystyle |||.|_cf||_{2,c}^2=\sum_{j,k,l}a_{j,k,l}<||.||_c^2\psi_{j,k,l},f>\qquad\qquad\qquad\qquad\qquad\qquad\qquad\qquad\qquad\qquad$$
$$=\displaystyle\sum_{j,k,l}a_{2j,k,l}<||.||_c^2\psi_{2j,k,l},f>+\sum_{j,k,l}a_{2j+1,k,l}<||.||_c^2\psi_{2j+1,k,l},f>.$$

\noindent Thus :
$$\displaystyle |||.|f||_{2,c}^2=-\sum_{j,k,l}(j+1)a_{2j,k,l}a_{2j+2,k,l}||\psi_{2j+2,k,l}||_c^2+\sum_{j,k,l}a_{2j,k,l}^2(\frac{m}{2}+k+2j)||\psi_{2j,k,l}||_c^2$$
$$\quad\qquad \qquad  \displaystyle -\sum_{j,k,l}(\frac{m}{2}+k+j-1)a_{2j,k,l}a_{2j-2,k,l}||\psi_{2j-2,k,l}||_c^2
-\sum_{j,k,l}(j+1)a_{2j+1,k,l}a_{2j+3,k,l}||\psi_{2j+3,k,l}||_c^2$$
$$\quad \qquad \qquad   +\displaystyle\sum_{j,k,l}a_{2j+1,k,l}^2(\frac{m}{2}+k+2j+1)||\psi_{2j+1,k,l}||_c^2-\displaystyle\sum_{j,k,l}a_{2j+1,k,l}a_{2j-1,k,l}(\frac{m}{2}+k+j)||\psi_{2j-1,k,l}||_c^2.$$

\noindent With the same manner, we find :
$$\displaystyle |||.|_c^2\mathcal{F}_{\pm}(f)||_{2,c}^2=
\sum_{j,k,l}(j+1)a_{2j,k,l}a_{2j+2,k,l}||\psi_{2j+2,k,l}||_c^2+\sum_{j,k,l}a_{2j,k,l}^2(\frac{m}{2}+k+2j)||\psi_{2j,k,l}||_c^2$$
$$\qquad\qquad \qquad +\displaystyle\sum_{j,k,l}(\frac{m}{2}+k+j-1)a_{2j,k,l}a_{2j-2,k,l}||\psi_{2j-2,k,l}||_c^2
+\displaystyle\sum_{j,k,l}(j+1)a_{2j+1,k,l}a_{2j+3,k,l}||\psi_{2j+3,k,l}||_c^2$$
$$\qquad\qquad \qquad +\displaystyle\sum_{j,k,l}a_{2j+1,k,l}^2(\frac{m}{2}+k+2j+1)||\psi_{2j+1,k,l}||_c^2+\sum_{j,k,l}a_{2j+1,k,l}a_{2j-1,k,l}(\frac{m}{2}+k+j)||\psi_{2j-1,k,l}||_c^2.$$

\noindent Now, by collecting the two equalities  we get :

\begin{equation}
|||.|_cf||_{2,c}^2+|||.|_c^2\mathcal{F}_{\pm}f||_{2,c}^2=2(\sum_{j,k,l}a_{2j,k,l}^2(\frac{m}{2}+k+2j)||\psi_{2j,k,l}||_c^2\qquad\qquad\qquad
\end{equation}
$$\qquad\qquad\qquad\qquad\qquad\qquad+\sum_{j,k,l}a_{2j+1,k,l}^2(\frac{m}{2}+k+2j+1)||\psi_{2j+1,k,l}||_c^2)$$
$$\qquad\qquad\qquad\qquad\geq 2\displaystyle\sum_{j,k,l}a_{j,k,l}^2(\frac{m}{2}+k+2j)||\psi_{j,k,l}||_c^2 $$
$$\geq m ||f||_{2,c}^2.\qquad $$
\\
\noindent Let $f\in L_2(\mathbb{R}^m)\otimes Cl_{0,m}$ and $k>0.$  We define the function $f_k$ by :
\begin{equation}
f_k(x):=f(kx).
\end{equation}
\noindent Then, we have :
\begin{equation}
\displaystyle
|||.|f_k||_{2,c}^2=\int_{\mathbb{R}^m}||x||_c^2||f_k(x)||_c^2dx=\int_{\mathbb{R}^m}||x||_c^2||f(kx)||_c^2dx \qquad\qquad
\end{equation}
$$\quad \displaystyle =k^{m-2}\int_{\mathbb{R}^m}||x||_c^2||f(x)||_c^2dx=k^{m-2}|||.|f||_{2,c}^2.$$
\noindent By lemma 3.3, we find :
\begin{equation}
|||.|\mathcal{F}_{\pm} (f_k)||^2_{2,c}=\displaystyle\int_{\mathbb{R}^m}||\lambda||_c^2||\mathcal{F}_{\pm}(f_k)(\lambda)||_c^2d\lambda\quad
\end{equation}
$$\qquad\qquad\qquad\qquad=k^{2m}\displaystyle\int_{\mathbb{R}^m}||\lambda||_c^2||\mathcal{F}_{\pm}(f)(k^{-1}\lambda)||_c^2d\lambda$$
$$\qquad\qquad\qquad\quad=k^{2+m}\displaystyle\int_{\mathbb{R}^m}||\lambda||_c^2||\mathcal{F}_{\pm}(f)(\lambda)||_c^2d\lambda$$
$$\qquad=k^{m+2}|||.|\mathcal{F}_{\pm} (f)||^2_{2,c}.$$
Using relation (4.10), we obtain :
\begin{equation}
k^{-2}|||.|f||_{2,c}^2+k^{2}|||.|\mathcal{F}_{\pm}(f|)|_{2,c}^2\geq m||f||_{2,c}^2.
\end{equation}
Lemma 4.3, completes the proof of inequality.\\

\noindent Now, we treat the case of equality for a function $f$ almost everywhere equal to $c e^{-p||x||_c^2}$ for some positive constant $p$.
We check the case  $p=\frac{1}{2}.$ Then, we generalize it for some positive $p.$ \\

\noindent Let $f_1=ce^{\frac{-||x||_c^2}{2}}$, then 
$$|||.|f_1||_{2,c}^2=\displaystyle\int_{\mathbb{R}^m}||x||_c^2c^2e^{-||x||_c^2}dx.$$
By Theorem 3.5, we have :
$$|||.|\mathcal{F}_\pm (f_1)||_{2,c}^2=|||.|f_1||_{2,c}^2.$$
Thus
$$|||.|f_1||_{2,c}^2|||.|\mathcal{F}_\pm  f_1||_{2,c}^2=\left(\displaystyle\int_{\mathbb{R}^m}||x||_c^2c^2e^{-||x||_c^2}dx\right)^2.$$
Since $\partial_x (x)=-m$, $x^2=-||x||_c^2$ and $\partial_x( \frac{1}{2}e^{-||x||_c^2})=x e^{-||x||_c^2},$ an integration by parts shows that :
$$\displaystyle\int_{\mathbb{R}^m}||x||_c^2e^{-||x||_c^2}dx=-\frac{m}{2}\displaystyle\int_{\mathbb{R}^m}e^{-||x||_c^2}dx.$$
Hence
\begin{equation}
|||.|f_1||_{2,c}^2|||.|\mathcal{F}_\pm (f_1)||_{2,c}^2=\frac{m^2}{4}||f_1||_{2,c}^4.
\end{equation}  
\noindent For the general case, we just find  relations between the norm $|||.|f||_{2,c}^2$ and $|||.|f_1||_{2,c}^2$,  also between
 $|||.|\mathcal{F}_\pm (f)||_{2,c}^2$ and  $|||.|\mathcal{F}_\pm (f_1)||_{2,c}^2.$\\  
 
\noindent  Put $f(x)=ce^{-p||x||_c^2}$. We recall that the Clifford-Fourier transform  for such $f$ is defined by 
$$\mathcal{F}_\pm (f)(x)={(2\pi)}^\frac{-m}{2}\int_{\mathbb{R}^m} K_{\pm}(y,x)ce^{-p||y||_c^2}dy.$$
Since $f(x)=f_1(\sqrt{2p}x)$, for all $ x\in\mathbb{R}^m,$ lemma 3.3  leads to
$$\mathcal{F}_\pm (f)(x)
=(\sqrt{2p})^{-m}\mathcal{F}_\pm (f_1)(\frac{x}{\sqrt{2p}}).\qquad\qquad\qquad$$
Consequently
$$|||.|\mathcal{F}_\pm( f)||_{2,c}^2=\displaystyle\int_{\mathbb{R}^m}||x||_c^2||\mathcal{F}_\pm f(x)||_c^2dx$$
$$\qquad\qquad \qquad\qquad\qquad=(2p)^{-m}\displaystyle\int_{\mathbb{R}^m}||x||_c^2||\mathcal{F}_\pm f_1(\frac{x}{\sqrt{2p}})||_c^2dx.$$
By a change of variable, we have :
$$|||.|\mathcal{F}_\pm (f)||_{2,c}^2=(\sqrt{2p})^{-m+2}|||.|\mathcal{F}_\pm (f_1)||_{2,c}^2.$$
By another change of variable,  we obtain 
$$|||.| f||_{2,c}^2=(\sqrt{2p})^{-m-2}|||.|f_1||_{2,c}^2.$$
The required  equality is deduced by relation (4.15).

\end{proof}

\section{Hardy's theorem}
In this section, we prove Hardy's theorem for the Clifford-Fourier transform.

\begin{theorem}
Let p and q are positive constants and m is an  even integer  upper than two.
Assume $f\in B(\mathbb{R}^m)\otimes Cl_{0,m}$ such that : 
\begin{equation}
||f(x)||_c\leq C e^{-p||x||_c^2},\quad x \in \mathbb{R}^m
\end{equation}
and 
\begin{equation}
||\mathcal{F}_\pm(f)(\lambda)||_c\leq C e^{-q||\lambda||_c^2},\quad \lambda \in \mathbb{R}^m,
\end{equation}
for some positive constant $C.$ Then, three cases   can occur :\\

	\noindent i) If   $pq>\frac{1}{4},$ then $f=0.$\\
	 ii) If   $pq=\frac{1}{4},$ then $f(x)=A e^{-\frac{||x||_c^2}{2}},$ where $A$ is a constant.\\
	iii) If   $pq<\frac{1}{4},$ then there are many f. 
\end{theorem}
\begin{proof}
The functions $\{\psi_{j,k,l}\}$ defined by (4.1) and (4.2) give an infinite number of examples for iii).\\
It is well known that by scaling (3.5), we may assume $p=q$ without loss of generality. The proof of i) is 
a simple  deduction of ii).\\

\noindent Assume $p=q=\frac{1}{2}.$ Let $f\in B(\mathbb{R}^m)\otimes Cl_{0,m}$ satisfying (5.1), we have for $\lambda \in \mathbb{R}^m\otimes\mathbb{C}$

$$||\mathcal{F}_{\pm}(f)(\lambda)||_c\leq{(2\pi)}^\frac{-m}{2}\int_{\mathbb{R}^m} ||K_{\pm}(x,\lambda)||_c||f(x)||_cdx$$
$$\qquad\qquad\qquad\leq C{(2\pi)}^\frac{-m}{2}\displaystyle\int_{\mathbb{R}^m} ||K_{\pm}(x,\lambda)||_ce^{-\frac{||x||_c^2}{2}}dx.$$

\noindent Applying lemma 3.1, we get :
$$||\mathcal{F}_{\pm}(f)(\lambda)||_c\leq C{(2\pi)}^\frac{-m}{2}\displaystyle\int_{\mathbb{R}^m} e^{||x||_c||\lambda||_c-\frac{||x||_c^2}{2}}dx$$
$$\qquad \qquad\qquad\qquad\leq e^{\frac{||\lambda||_c^2}{2}}C{(2\pi)}^\frac{-m}{2}\displaystyle\int_{\mathbb{R}^m}e^{\frac{-\left(||x||_c-||\lambda||_c\right)^2}{2}}dx.$$
Thus
\begin{equation}
||\mathcal{F}_{\pm}(f)(\lambda)||_c\leq C'e^{\frac{||\lambda||_c^2}{2}},
\end{equation}
where C' is a  positive constant.\\
As $\mathcal{F}_{\pm}(f)$ is an entire function satisfying relation (5.2) and (5.3), lemma 2.1 in\rm[8] allows to express $\mathcal{F}_{\pm}(f)$ as follows :
$$\mathcal{F}_{\pm}(f)(x)=A e^{-\frac{||x||_c^2}{2}},$$ with A is a constant.\\
\noindent Theorem 3.5 completes the proof.

\end{proof}

\end{document}